\documentclass[12pt,bezier]{article}
\usepackage{times}
\usepackage{booktabs}
\usepackage{pifont}
\usepackage{floatrow}
\usepackage{caption}
\usepackage{mathrsfs}
\usepackage[fleqn]{amsmath}
\usepackage{amsfonts,amsthm,amssymb,mathrsfs,bbding}
\usepackage{txfonts}
\usepackage{graphics,multicol}
\usepackage{graphicx}
\usepackage{color}
\usepackage{enumerate}
\usepackage{caption}
\usepackage{cite}
\usepackage{latexsym,bm}
\usepackage{indentfirst}
\usepackage{mathtools}
\usepackage{diagbox}
\evensidemargin=\oddsidemargin\topmargin=-1.5cm

\newtheorem{thm}{Theorem}[section]

\newtheorem{lem}{Lemma}[section]
\newtheorem{cor}{Corollary}[section]

\newtheorem{remark}{Remark}[section]

\theoremstyle{definition}

\addtocounter{section}{0}

\begin{document}
\title{Some new bounds for the signless Laplacian energy of a graph\footnote{This work is supported
by the National Natural Science Foundation of China (Grant nos. 11971274, 11531011, 11671344).}}

\author{
{\small Peng Wang,\ \ Qiongxiang Huang\footnote{Corresponding author.
\newline{\it \hspace*{5mm}Email addresses:} huangqx@xju.edu.cn (Q.X. Huang).}\;\ \ }\\[2mm]
\footnotesize College of Mathematics and Systems Science, Xinjiang University, Urumqi, Xinjiang 830046, China\\}
\date{ }
\maketitle
{\flushleft\large\bf Abstract}
For a simple graph $G$ with $n$ vertices, $m$ edges and signless Laplacian eigenvalues $q_{1} \geq q_{2} \geq \cdots \geq q_{n} \geq 0$, its the signless Laplacian energy $QE(G)$ is defined as $QE(G) = \sum_{i=1}^{n}|q_{i} - \bar{d} |$, where $\bar{d} = \frac{2m}{n}$ is the average vertex degree of $G$.
In this paper, we obtain two  lower bounds ( see Theorem \ref{thm-1} and Theorem \ref{thm-2} ) and one upper bound for $QE(G)$ ( see Theorem \ref{thm-3} ), which improve some known bounds of $QE(G)$, and moreover, we determine the corresponding extremal graphs that achieve our bounds. By subproduct, we also get some bounds for $QE(G)$ of regular graph $G$.

\vspace{0.1cm}
\begin{flushleft}
\end{flushleft}
\section{Introduction}

Let $G(V,E)$ be a simple graph with vertex set $V(G)=\{v_{1},v_{2}, \ldots, v_{n}\}$ and edge set $E(G)=\{e_{1},e_{2}, \ldots, e_{m}\}$. The adjacency matrix $A(G) = (a_{ij})$ of $G$ is defined by $a_{ij} = 1$ if $i \sim j$, and $a_{ij} = 0$ otherwise. The eigenvalues of $G$ are those of $A(G)$, which are denoted by $\lambda_{1} \geq \lambda_{2} \geq \cdots \geq \lambda_{n}$.  I. Gutman in \cite{Gutman1} introduced the notion of energy of $G$:
$E(G) = \sum_{i=1}^{n}| \lambda_{i}|$, which is received great attention and has made great progress
 in both chemical and mathematical applications\cite{Li}.

The signless Laplacian matrix of $G$ is defined as $Q(G) = D(G) + A(G)$, where $D(G) = diag(d_{1},d_{2}, \ldots , d_{n})$ is the diagonal matrix of vertex degrees of $G$. The eigenvalues of $Q(G)$ are called the signless Laplacian eigenvalues of $G$ ( short for $Q$-eigenvalues ), which are denoted by $q_1\geq q_2\geq\cdots \geq q_n$, and all its $Q$-eigenvalues along with their multiplicities consist of the spectrum called $Q$-spectrum and denoted by $Spec_{Q}(G)$. The Laplacian matrix of $G$ is defined as $L(G) = D(G) - A(G)$, its $L$-eigenvalues and $L$-spectrum are similarly defined.

The motivation for Laplacian energy comes from graph energy \cite{Gutman1, Li}. The Laplacian energy of a graph $G$ as put forward by Gutman and Zhou \cite{Gutman2} is defined as $LE(G) = \sum_{i=1}^{n}|  \mu_{i} - \frac{2m}{n}  |.$ This equation is an extension of the concept of graph energy.
Similar to the Laplacian energy, the signless Laplacian energy of a graph $G$ as put forward by Ganie, Hilal and Pirzada \cite{Ganie1} is defined as
$QE(G) = \sum_{i=1}^{n}|  q_{i} - \frac{2m}{n}  |.$
Particularly, if $G$ is a regular graph, then $q_i=\lambda_i+\frac{2m}{n}$. Thus $QE(G)=\sum_{i=1}^{n}| q_{i}-\frac{2m}{n} |=\sum_{i=1}^{n}|\lambda_{i}|=E(G)$. However, $LE(G)\neq E(G)$ for regular graph $G$.
The Laplacain energy and signless Laplacian energy are applied not only to theoretical organic chemistry \cite{Gutman3}, but also to image processing and information theory \cite{Radenkovi}.

{\flushleft\bf There are some results related to the lower bounds of $QE(G)$.}
For example, in 2017, Hilal A. Ganie et al. give a lower bound for $QE(G)$ in Theorem 3.3 in \cite{Ganie1}:
\begin{equation}\label{EQ-eq-1}\\\\QE(G) \geq 2(\frac{M_{1}}{m}-\frac{2m}{n})\end{equation}
with equality holds if and only if $G \cong K_{1,n-1}$, where $n=|V(G)|$, $m=|E(G)|$ and $M_{1} = M_{1}(G) = \sum_{i=1}^{n}d_{i}^{2}$ is
the first Zagreb index proposed by Gutman and Trinajsti\'{c} \cite{Trinajsti}.
The other two lower bounds for $QE(G)$ in Corollary 3.2 and Theorem 3.1 in \cite{Ganie1} are
\begin{equation}\label{EQ-eq-2}\\\\QE(G) \geq 2\Delta+2-\frac{4m}{n}\end{equation}
with equality holds if and only if $G \cong K_{1,n-1}$,
and
\begin{equation}\label{EQ-eq-3}\\\\QE(G) \geq \Delta + \delta +\sqrt{(\Delta-\delta)^{2}+4\Delta}-\frac{4m}{n}\end{equation}
with equality holds if and only if $G \cong K_{1,n-1}$.
In 2018, Hilal A. Ganie, et al. give two lower bounds for $QE(G)$ in Theorem 2.3 and Theorem 2.10 in \cite{Ganie2}:
\begin{equation}\label{EQ-eq-4}\\\\
QE(G) \geq\left\{\begin{array}{ll}{\frac{2 M_{1}}{m}+2 d_{2}-\frac{8m}{n}} &\mbox{if $v_{1}\nsim v_{2}$}, \\
{\frac{2 M_{1}}{m}+\Delta+d_{2}-\sqrt{(\Delta-d_{2})^{2}+4}-\frac{8m}{n}}  &\mbox{if $v_{1}\thicksim v_{2}$}
\end{array}\right.
\end{equation}
with equality holds if and only if $G \cong K_{n-2,2}$, where $v_{1}$ and $v_{2}$ are the vertices of the largest and second largest degree in $G$, and
$QE(G) \geq \frac{8 m}{n}-2 \delta$ if $G$ is bipartite graph
and otherwise,
\begin{equation}\label{EQ-eq-5}\\\\
Q E(G) \geq\left\{\begin{array}{ll}{\frac{8 m}{n}-2 \delta-2 d_{n-1}} &\mbox{if $v_{n}\thicksim v_{n-1}$}, \\
{\frac{8 m}{n}-(2d_{n-1}+\Delta+\delta-\sqrt{(\Delta-\delta)^{2}+4})} &\mbox{if $v_{n}\nsim v_{n-1}$}
\end{array}\right.
\end{equation}
with equality holds if and only if $G \cong K_{1,2}$, where $v_{n}$ and $v_{n-1}$ are the vertices of the smallest and second smallest degree in $G$.

{\flushleft\bf There are some results related  to the upper bounds of $QE(G)$.}
For example, in 2011, Nair Abreua, et al. give two upper bounds for $QE(G)$ in Theorem 5 in \cite{Abreu}:
\begin{equation}\label{EQ-eq-21}\\\\
QE(G) \leq 4m(1-\frac{1}{n})
\end{equation}
with equality holds if and only if  either $G$ is a null graph (that is a graph with n vertices and without edges) or $G$ is a graph with only one edge plus $n-2$ isolated vertices, and
\begin{equation}\label{EQ-eq-22}\\\\
QE(G) \leq \bigg[1+\sqrt{\frac{m}{2}-(\frac{2m}{n}-1)}\bigg]\sqrt{2(M_{1}-2m)}.
\end{equation}
In 2015, Rao Li gives an upper bound for $QE(G)$ in Theorem 2.5 in \cite{Li2}:
\begin{equation}\label{EQ-eq-23}\\\\
QE(G) \leq \frac{2m}{n-1} + n - 2+\sqrt{(n-2)(\frac{2m^{2}}{n-1}+\frac{8m\Delta-4m^{2}}{n}+mn-4)}
\end{equation}
with equality holds if and only if $G \cong K_{2}.$
In 2017, Ganie and Pirzada give an upper bound for $QE(G)$ in Theorem 4.1 in \cite{Ganie1}:
\begin{equation}\label{EQ-eq-24}\\\\
QE(G) \leq 2(2m + 1 - \Delta - \frac{2m}{n})
\end{equation}
with equality holds if and only if $G \cong K_{1,n-1}$.

In this paper, we obtain some new bounds for $QE(G)$ which improve some known results. Moreover, we determine the extremal graphs that achieve our bounds of $QE(G)$. As subproduct, we also obtain some bounds for $QE(G)$ of regular graph $G$, which can also be viewed as the bounds of $E(G)$, additionally, we also characterize the corresponding extremal graphs.
The paper is organized as follows. In Section 2, we list some previously known results. In Section 3, we first give new bounds of  $QE(G)$ and determine their extremal graphs, next we give some bounds for $QE(G)$ of regular graph $G$ and determine the corresponding extremal graphs. In Section 4,  we give some examples and  tables, from which one can see that  our bounds are closer to exact values of $QE(G)$.

\section{Preliminaries}
In this section, we will cite some result related with $Q$-eigenvalue of a graph $G$ for the later use.
\begin{lem}[{\cite{Cvetkovic}}]\label{lem-9}
For any graph $G$ , the multiplicity of the Q-eigenvalue 0 is equal to the number of components that is bipartite.
\end{lem}

\begin{lem}[{\cite{Brouwer}}]\label{lem-15}
(Interlacing Theorem) If $M$ is a real symmetric $n\times n$ matrix, let $\lambda_{1}(M)\geq \lambda_{2}(M)\geq \cdots \geq \lambda_{n}(M)$ denote its eigenvalues in nonincreasing order.
Suppose $A$ is a real symmetric $n\times n$ matrix and  $B$ is a principal submatrix of $A$ with order $m\times m$. Then, for $i=1,2,...,m$,  $\lambda_{n-m+i}(A)\leq\lambda_{i}(B)\leq\lambda_{i}(A)$.
\end{lem}

\begin{lem}[{\cite{Feng}}]\label{lem-4}
Let $G$ be a graph with $n$ vertices and $m$ edges and let $q_{1}$ be its largest $Q$-eigenvalue. Then
$q_{1} \geq \frac{4m}{n}$
with equality if and only if $G$ is a regular graph.
\end{lem}

\begin{lem}[{\cite{Oboudi}}]\label{lem-2}
Let $n \geq 1$ be an integer and $a_{1} \geq a_{2} \geq \cdots \geq a_{n}$ be some nonnegative real numbers. Then
$ \sum_{i=1}^{n}a_{i}(a_{1} + a_{n}) \geq \sum_{i=1}^{n}a_{i}^{2} + na_{1}a_{n}$
with equality holds if and only if $a_{1} = \cdots = a_{s}$ and $a_{s+1} = \cdots = a_{n}$ for some $s\in \{1,\ldots,n\}$.
\end{lem}

\begin{lem}[{\cite{Lima}}]\label{lem-5}
Let $G$ be a graph with $n$ vertices and $m$ edges. Then
$q_{min}(G) \leq \frac{2m}{n} - 1$
with equality if and only if $G$ is a complete graph.
\end{lem}

\begin{lem}[{\cite{Brouwer}}]\label{lem-6}
Let $G$ be a graph and $q_{1}$ be its $Q$-spectral radius. Then the following hold:\\
(1) If $G$ is connected, then the multiplicity of $q_{1}$ is one;\\
(2) For every eigenvalue $q_{i}$ of $G$, $|q_{i}| \leq q_{1}$.
\end{lem}

\begin{lem}[{\cite{Wang}}]\label{lem-7}
Let $G$ be a graph with $n$ vertices and $m$ edges and let $q_{1}$ be its largest $Q$-eigenvalue. Then
$q_{1}\leq \frac{2m + \sqrt{m(n^{3}-n^{2}-2mn+4m)}}{n}$
with equality holds if and only if $G$ is a complete graph.
\end{lem}

\begin{lem}[{\cite{Oboudi}}]\label{lem-3}
Let $\alpha$, $x$, $y$ and $\beta$ be some positive real numbers such that $0<\alpha \leq x \leq y \leq \beta$. Then
$\frac{\sqrt{\alpha \beta}}{\alpha+\beta} \leq \frac{\sqrt{x y}}{x+y}$
with equality holds if and only if $x=\alpha$ and $y=\beta$.
\end{lem}

\begin{lem}[{\cite{Milovanovic}}]\label{lem-8}
Let $G$ be a connected graph with $n$ vertices and $m$ edges. Then
$M_{1}\geq \frac{4m^{2}}{n} + \frac{1}{2}(\Delta-\delta)^{2}$
with equality holds if and only if $G$ is a regular graph.
\end{lem}

\begin{lem}[{\cite{Cvetkovic}}]\label{lem-18}
Let $G$ be a graph with $n$ vertices and $m$ edges and let $q_{1}$ be its largest $Q$-eigenvalue. Then
$2\delta\leq q_{1}\leq 2\Delta$.
For a connected graph $G$, equality holds in either place if and only if $G$ is regular.
\end{lem}

\begin{lem}[{\cite{Axler}}]\label{lem-17}
A connected regular graph with exactly three distance eigenvalues is strong regular graph.
\end{lem}

\begin{lem}[{\cite{Fath-Tabar}}]\label{lem-19}
Let G be a connected graph with $n$ vertices and $m$ edges. Then
$M_{1}\leq\frac{4m^2}{n}+\frac{n}{4}(\Delta-\delta)^2$ with equality holds if and only if $G$ is a regular graph.
\end{lem}

\begin{lem}[{\cite{Brouwer}}]\label{lem-16}
Let $G$ and $H$ be two disjoint graphs. Assume that $Spec(G)=\{\ \lambda_1,..., \lambda_n\ \}$ and $Spec(H) =\{\ \mu_1,..., \mu_m\ \}$. Then $Spec(G\Box H)=\{\ \lambda_i+\mu_j;\ i=1,...,n,\ and\  j=1,...,m\ \}$.
\end{lem}

\section{Main result}
In this section, we focus to give new lower and upper bounds of $QE(G)$ and characterize the corresponding extremal graphs. Moreover, we apply these bounds to the regular graph and also characterize the corresponding extremal graphs.

\begin{lem}\label{lem-10}
Let $q_{1} \geq q_{2} \geq \cdots \geq q_{n} \geq 0$ be the $Q$-eigenvalues of $G$ and $M_{1}$ be the first Zagreb index of $G$. We have
$\sum_{i=1}^{n}q_{i} = 2m$ and  $\sum_{i=1}^{n}q_{i}^{2} = 2m + M_{1}$, where $m$ is the number of edges of $G$.
\end{lem}
\begin{proof}
Let $d_i$ be the degree of the vertex $v_i\in V(G)$. It is clear that $\sum_{i=1}^{n}q_{i} = tr(Q(G))  = \sum_{i=1}^{n}d_{i} = 2m $. Therefore,
$$\begin{array}{ll}
\sum_{i=1}^{n}q_{i}^{2}& = tr(Q(G)^{2})\\
& = tr((D(G) + A(G))^{2})\\
&= tr(D(G)^{2})+2tr(A(G)D(G))+tr(A(G)^{2})\\
&= M_{1}+ 2m.
\end{array}
$$
It follows our result.
\end{proof}

\begin{lem}\label{lem-11}
Let $G$ be a connected bipartite $r$-regular graph with $n$ vertices and $m$ edges. Assume that
$$Spec_{Q}(G)=\{2r, [r+1]^{a}, [r-1]^{b}, 0\},$$
where $a$ and $b$ are some non-negative integers. Then $a=b=r=\frac{n}{2}-1$ and $G \cong K_{r+1,r+1}\backslash F$, where $F$ is perfect matching of the bipartite graph $K_{r+1,r+1}$.
\end{lem}
\begin{proof}
By Lemma \ref{lem-10}, we have $2r+a(r+1)+b(r-1)=2m$ and $(2r)^{2}+a(r+1)^{2}+b(r-1)^{2}=2m+M_{1}$. Since $n=2+a+b$ and $m=\frac{nr}{2}$, we have $a=b=r=\frac{n}{2}-1$. Since $G$ is connected bipartite $r$-regular graph, we have $G \cong K_{r+1,r+1}\backslash F$.
\end{proof}

\begin{lem}\label{lem-12}
Let $G$ be a connected $r$-regular graph with $n$ vertices and $m$ edges. Assume that
$$Spec_{Q}(G)=\{2r, [r+1]^{a}, [r-1]^{b}\},$$
where $a$ and $b$ are some non-negative integers. Then $a=0$, $b=r=n-1$ and $G \cong K_{r+1}$.
\end{lem}
\begin{proof}
By Lemma \ref{lem-10}, we have $2r+a(r+1)+b(r-1)=2m$ and $(2r)^{2}+a(r+1)^{2}+b(r-1)^{2}=2m+M_{1}$. Since $n=1+a+b$ and $m=\frac{nr}{2}$, we have $a=0$, $b=r=n-1$. Since $G$ is connected $r$-regular graph, we have $G \cong K_{r+1}$.
\end{proof}

\begin{lem}\label{lem-13}
Let $G$ be a non-connected $r$-regular graph with $n$ vertices and $m$ edges. Assume that
$$Spec_{Q}(G)=\{[2r]^{s'},[0]^{s-s'},[r+1]^a,[r-1]^b\},$$
where $a$, $b$, $s$, $s'$ are some non-negative integers, $n=s+a+b$ and $s> s'>1$.
Then $a=r(s-s')$, $b=rs'$, $r=\frac{n}{s}-1$ and $G \cong gK_{r+1} \bigcup h(K_{r+1,r+1}\backslash F)$, where $r\geq 2$, $g=2s'-s$, $h=s-s'$ and $F$ is perfect matching of the bipartite graph $K_{r+1,r+1}$.
\end{lem}
\begin{proof}
By Lemma \ref{lem-10}, we have $2rs'+a(r+1)+b(r-1)=2m$ and $(2r)^{2}s'+a(r+1)^{2}+b(r-1)^{2}=2m+M_{1}$. Since $n=s+a+b$ and $m=\frac{nr}{2}$, we have $a=r(s-s')$, $b=rs'$ and $r=\frac{n}{s}-1$.
Since $G$ is a non-connected $r$-regular graph, its $Q$-spectral radius $q_1=2r$ with multiplicity $s'$, $G$ has exactly $s'$ connected components, say $G_1$,..., $G_{s'}$. Let $n_{i}$ and $m_{i}$ be the numbers of the vertices and edges of $G_{i}$, respectively, where  $n=\sum_{i=1}^{s'}n_{i}$ and  $m=\sum_{i=1}^{s'}m_{i}$.
We start to analyze  components $G_{i}$.

If there is $G_i$ without $Q$-eigenvalue $0$, then
\begin{equation}\label{EQ-eq-6}\\\\
Spec_{Q}(G_{i})=\{2r,[r+1]^{a_i},[r-1]^{b_i}\},
\end{equation}
where $0\le a_i\le a$ and $0\le b_i\le b$. By Lemma \ref{lem-12}, we have $G_{i} \cong K_{r+1}$ and $a_i=0$, $b_i=r=n_{i}-1$.

If there is $G_i$ with $Q$-eigenvalue $0$ and multiplicity $m_{G_i}(0)=s_{i}>0$. Then the $Q$-spectrum of $G_{i}$ has three choices:
$Spec_{Q}(G_{i})=\{2r,[0]^{s_{i}},[r+1]^{a_i}\}$, where $a_i=n_{i}-s_{i}-1$, $Spec_{Q}(G_{i})=\{2r,[0]^{s_{i}},[r-1]^{b_i}\}$, where $b_i=n_{i}-s_{i}-1$ or  \begin{equation}\label{EQ-eq-7}\\\\
Spec_{Q}(G_{i})=\{2r,[0]^{s_{i}},[r+1]^{a_i},[r-1]^{b_i}\},
\end{equation}
where $0\le a_i\le a$ and $0\le b_i\le b$.
If the first situation appears, then, by Lemma \ref{lem-10}, we have $2r+(r+1)(n_{i}-s_{i}-1)=2m_{i}$ and $(2r)^{2}+(r+1)^{2}(n_{i}-s_{i}-1)=2m_{i}+M_{1}(G_i)$. Since $2m_{i}=n_{i}r$ and $M_{1}(G_i)=n_{i}r^{2}$, we have $r=0$ or $1$, which contradicts $r\geq2$.
If the second situation appears, then $2r+(r-1)(n_{i}-s_{i}-1)=2m_{i}$ and $(2r)^{2}+(r-1)^{2}(n_{i}-s_{i}-1)=2m_{i}+M_{1}(G_{i})$. Since $2m_{i}=n_{i}r$ and $M_{1}(G_{i})=n_{i}r^{2}$, we have $r=n_{i}-1$ and $(2-n_{i})s_{i}=0$. Since $r=n_{i}-1\geq2$, we have $n_{i}\geq3$. Therefore, $s_{i}=0$, a contradiction.
If the last situation appears, then $s_{i}=1$ and $G_{i}$ is a bipartite graph since $m_{G_i}(0)=s_{i}$ equals the number of even components of $G_{i}$
by Lemma \ref{lem-9}.
By Lemma \ref{lem-11}, we have $G_{i} \cong K_{r+1,r+1}\backslash F$ and $a_i=b_i=r$.

From above discussions, we may assume that $G$ contains exactly $g$ $(0\le g\le s')$ copies of $K_{r+1}$, say $G_1$, $G_2$,...,$G_g$, and $h$ $(0\le h\le s')$ copies of $K_{r+1,r+1}\backslash F$, say $G_{g+1}$, $G_{g+2}$,...,$G_{g+h}$.
Since $b_{j}=r$ for $1\le j\le g$ in (\ref{EQ-eq-6}) and $s_{i}=1$, $a_{i}=b_{i}=r$ for $g+1\le i\le g+h$ in (\ref{EQ-eq-7}), comparing with their spectra we have $g+h=s'$, $s_ih=s-s'$, $a_ih=a$ and $b_ih+b_jg=b$. Therefore, $h=\frac{a}{r}=s-s'$ and $g=s'-h=2s'-s$.
It follows that $G \cong gK_{r+1} \bigcup h(K_{r+1,r+1}\backslash F)$, where $r\geq2$, $g=2s'-s$ and $h=s-s'$.

We complete this proof.
\end{proof}

A graph $G$ is called $DQS$, if for any $H$, we have $H\cong G$ whenever  $Spec_Q(H)=Spec_Q(G)$. The proof of Lemma \ref{lem-13} implies the following result.

\begin{cor}\label{cor-1}
$G\!=\!gK_{r+1}\! \bigcup \!h(K_{r+1,r+1}\backslash F)$ is $DQS$-graph. Particularly,
$K_{r+1}$ and $K_{r+1,r+1}\backslash F$ are $DQS$-graph.
\end{cor}

\begin{lem}\label{lem-14}
A simple connected graph $G$ has exactly two distinct $Q$-eigenvalues if and only if $G\cong K_{n}$.
\end{lem}
\begin{proof}
Note that $Spec_{Q}(K_{n})=\{2n-2,[n-2]^{n-1}\}$, the sufficiency holds.

Now suppose that $G$ has two distinct $Q$-eigenvalues $\alpha>\beta\geq0$ and $G$ is not a complete graph. Since $G$ is connected, we have $Spec_{Q}(G)=\{\alpha,[\beta]^{n-1}\}$. Since $G$ is not a complete graph, then there exist $u,v\in V(G)$ such that $uv\notin E(G)$,  where $d(u)\le d(v)$. Thus signless Laplacian matrix $Q$ of $G$ contain a principal submatrix
$B=\left(
  \begin{array}{cc}
    d(u) & 0 \\
    0 & d(v) \\
  \end{array}
\right)$ that is induced by vertices $u$ and $v$. By Lemma \ref{lem-15}, we have $\beta= q_{n}(Q)\le \lambda_{2}(B)\le q_2(Q)=\beta$ and so $\beta=\lambda_{2}(B)=d(u)$. On the other hand, there exists $v'\in V(G)$ such that $uv'\in E(G)$. Thus signless Laplacian matrix $Q$ of $G$ contains a principal submatrix
$B'=\left(
  \begin{array}{cc}
    d(u) & 1 \\
    1 & d(v') \\
  \end{array}
\right)$ that is induced by vertices $u$ and $v'$. By simply calculation, we have $\lambda_2(B')=\frac{d(u)+d(v')-\sqrt{(d(u)-d(v'))^2+4}}{2}$. As the same reason as above, we have $\frac{d(u)+d(v')-\sqrt{(d(u)-d(v'))^2+4}}{2}=\lambda_2(B')=\beta=d(u)$, which leads to $(d(v')-d(u))^2=(d(u)-d(v'))^2+4$, a contradiction.
\end{proof}

For a graph $G$ with $n$ vertices and $m$ edges, let $q_{i}$  be the $Q$-eigenvalues of $G$ and  $\gamma_{i} = |q_{i}-\frac{2m}{n}|$, where  $i=1,2,...,n$, such that $\gamma_{1} \geq \gamma_{2} \geq \cdots \geq \gamma_{n}\ge 0$. Thus $ QE(G)=\sum_{i=1}^{n}\gamma_{i}$. Since $\gamma_n$ does not contribute to $ QE(G)$ if $\gamma_n=0$, without loss of generality, we always assume that $\gamma_{n}>0$ if we don't specifically state.

\begin{remark}\label{rem-1}
It needs to mention that, under the ordering of $\gamma_{1} \geq \gamma_{2} \geq \cdots \geq \gamma_{n}\ge 0$, the corresponding $Q$-eigenvalues $\{q_i\}$ is no long  to have the decreased order as usual. However, since the $Q$-spectral radius is no less than $\frac{4m}{n}$ according to Lemma \ref{lem-4}, we see that $q_1=\gamma_1+\frac{2m}{n}$ is really the $Q$-spectral radius of $G$.
\end{remark}

\begin{thm}\label{thm-1}
Let $G$ be a graph with $n \geq 2$ vertices and $m \geq 1$ edges, and $\gamma_{i} = |q_{i}-\frac{2m}{n}|$ defined above. Then
\begin{equation}\label{EQ-eq-8}\\\\QE(G) \geq 2\sqrt{(2m+M_{1}-\frac{4m^{2}}{n}) n} \cdot \frac{\sqrt{\gamma_{1}\gamma_{n}}}{\gamma_{1}+\gamma_{n}}\end{equation}
with equality holds if and only if $G \cong \frac{n}{2}K_{2}$ or $gK_{\frac{2m}{n} +1} \bigcup h(K_{\frac{2m}{n}+1,\frac{2m}{n}+1}\backslash F)$, where $g$ and $h$ are some non-negative integers, $\frac{2m}{n} \geq 2$ is an integer and $F$ is perfect matching of the bipartite graph $K_{\frac{2m}{n}+1,\frac{2m}{n}+1}$.
\end{thm}

\begin{proof}
By Lemma \ref{lem-10}, we have
\begin{equation}\label{EQ-eq-9}\sum_{i=1}^{n}\gamma_{i}^{2}=\sum_{i=1}^{n}|q_{i} - \frac{2m}{n}|^{2}=\sum_{i=1}^{n}q_{i}^{2} -\frac{4m}{n}\sum_{i=1}^{n}q_{i} + \sum_{i=1}^{n} (\frac{2m}{n})^{2} =2m+M_{1}-\frac{4m^{2}}{n}.\end{equation}
By Lemma \ref{lem-2}, we have \begin{equation}\label{EQ-eq-10}\begin{array}{ll}QE(G)=\sum_{i=1}^{n}\gamma_{i} &\geq \frac{\sum_{i=1}^{n}\gamma_{i}^{2} + n\gamma_{1}\gamma_{n}}{\gamma_{1} + \gamma_{n}} = \frac{2m+M_{1}-\frac{4m^{2}}{n} + n\gamma_{1}\gamma_{n}}{\gamma_{1} + \gamma_{n}}\\
&\geq \frac{2\sqrt{(2m+M_{1}-\frac{4m^{2}}{n}) n\gamma_{1}\gamma_{n}}}{\gamma_{1} + \gamma_{n}}=2\sqrt{(2m+M_{1}-\frac{4m^{2}}{n})n} \cdot \frac{\sqrt{\gamma_{1}\gamma_{n}}}{\gamma_{1}+\gamma_{n}}
\end{array}
\end{equation}
with the first equality holds if and only if $\gamma_{1} = \cdots = \gamma_{s}$ and $\gamma_{s+1} = \cdots = \gamma_{n}$ for some $1\le s\le n$, and  the second equality holds if and only if $2m+M_{1}-\frac{4m^{2}}{n} = n\gamma_{1}\gamma_{n}$.

Now suppose that (\ref{EQ-eq-10}) is an equality. We may assume there exists some $1\le s\le n$ such that $\gamma_{1} = \cdots = \gamma_{s} =\alpha \ge \gamma_{s+1} = \cdots = \gamma_{n} =\beta >0$ and
\begin{equation}\label{EQ-eq-11}\\\\2m+M_{1}-\frac{4m^{2}}{n} = n\alpha\beta>0.\end{equation}
From (\ref{EQ-eq-9}) and (\ref{EQ-eq-11}), we have  $s\alpha^{2}+(n-s)\beta^{2}  = n\alpha\beta$ and so \begin{equation}\label{EQ-eq-12}\\\\s(\alpha+\beta)(\alpha-\beta)=n\beta(\alpha-\beta).\end{equation}

Now we divide the following two cases.

{\flushleft\bf Case 1.} Suppose that $\alpha=\beta$.

In this case, we have  $|q_i-\frac{2m}{n}|=\gamma_{i}=\alpha$, i.e., $q_{i}=\alpha+\frac{2m}{n}$ or $-\alpha+\frac{2m}{n}$  for $i=1,2,...,n$. Therefore, we have $ Spec_{Q}(G)=\{[\alpha+\frac{2m}{n}]^{a}, [-\alpha+\frac{2m}{n}]^{b}\}$, where $a+b=n$.
By Lemma \ref{lem-10}, we have $a(\alpha+\frac{2m}{n})+b(-\alpha+\frac{2m}{n}) = 2m$, i.e., $(a-b)\alpha = 0$. Thus $\alpha = 0$, or $\alpha \neq 0$ and $a = b$.
If $\alpha = 0$, then $Spec_{Q}(G)=\{[\frac{2m}{n}]^{n}\}$, which contradicts the result of Lemma \ref{lem-5}.
If $\alpha \neq 0$ and $a = b$, then $ Spec_{Q}(G)=\{[\alpha+\frac{2m}{n}]^{\frac{n}{2}}, [-\alpha+\frac{2m}{n}]^{\frac{n}{2}}\}$. Clearly, $q_1=\alpha+\frac{2m}{n}$ is the $Q$-spectral radius of $G$. If $n=2$ then $q_1$  is simple and thus $G = K_2$ by Lemma \ref{lem-6} (1). Now we suppose that $n>2$. Again by Lemma \ref{lem-6} (1),  $G$ is disconnected and let $G_1$ be a component of $G$. We see that $G_1$ also has $Q$-spectral radius $q_1=\alpha+\frac{2m}{n}$. Thus $G$ has exactly $\frac{n}{2}$ components $G_1$,..., $G_\frac{n}{2}$, each of them has spectrum $Spec_{Q}(G_i)=\{[\alpha+\frac{2m}{n}]^{1}, [-\alpha+\frac{2m}{n}]^{1}\}$. Therefore, $G_i$ contains exactly two vertices and so $G_i=K_2$. It follows that $G \cong \frac{n}{2}K_{2}$.

{\flushleft\bf Case 2.} Suppose that $\alpha \neq \beta$.

In this case, there exists some $1\le s< n$ such that $|q_i-\frac{2m}{n}|=\gamma_{i}=\alpha$ for $i=1,2,...,s$ and $|q_j-\frac{2m}{n}|=\gamma_{j}=\beta$ for $j=s+1,...,n$. We have
$$q_{i}\in \{
\alpha+\frac{2m}{n},
-\alpha+\frac{2m}{n}\} \mbox{ for  $i=1,...,s$\ \  and \ \  }q_{j}\in \{ \beta+\frac{2m}{n},
 -\beta+\frac{2m}{n}\} \mbox{ for  $j=s+1,...,n$.  }
$$
According to Remark \ref{rem-1}, we claim that $q_1=\alpha+\frac{2m}{n}$ is $Q$-spectral radius of $G$. There exists $s\ge m_Q(q_1)=s'\ge 1$ and $a+b=n-s$ such that $G$ has $Q$-spectrum:
\begin{equation}\label{EQ-eq-13}
\\\\Spec_{Q}(G)=\{[\frac{2m}{n}+\alpha]^{s'},[\frac{2m}{n}-\alpha]^{s-s'},[\frac{2m}{n}+\beta]^a,[\frac{2m}{n}-\beta]^b\}.
\end{equation}
Since $q_1\ge \frac{4m}{n}$ by Lemma \ref{lem-4}, we have $\alpha\ge \frac{2m}{n}$.

First suppose that $G$ is connected graph. Then $s'=1$ and $q_1$ is simple. If $s\ge2$, then $q_{i}=-\alpha+\frac{2m}{n}$ for $i=2,3,...,s$. Thus,
if $\alpha > \frac{2m}{n}$, then $q_{i}=-\alpha+\frac{2m}{n}<0$ and it contradicts $q_{i} \geq 0$.
Therefore, $\alpha = \frac{2m}{n}$, and thus $q_{1}=\frac{4m}{n}$ and $q_i=0$ for $i=2,3,...,s$. By Lemma \ref{lem-4}, $G$ is a $\frac{2m}{n}$-regular graph.
Using (\ref{EQ-eq-11}), we have $\beta = 1+\frac{M_{1}}{2m} - \frac{2m}{n}=1+\frac{1}{2m}\cdot\frac{4m^2}{n}- \frac{2m}{n}=1$. Therefore, $q_{j}\in \{\frac{2m}{n}+1,\frac{2m}{n}-1\}$ for $j=s+1,...,n$ and $Spec_Q(G)=\{\frac{4m}{n},[0]^{s-1},[\frac{2m}{n}+1]^{a},[\frac{2m}{n}-1]^{b}\}$, where $s+a+b=n$. By Lemma \ref{lem-9}, the multiplicity of $Q$-eigenvalue $0$ equals the number of even components of $G$. It implies that $G$ is a bipartite graph and $s=2$. By Lemma \ref{lem-11}, we have $G \cong K_{\frac{2m}{n}+1,\frac{2m}{n}+1}\backslash F$. If $s=1$, then $s'=1$ and
$Spec_{Q}(G)=\{\frac{2m}{n}+\alpha,[\frac{2m}{n}+\beta]^{a},[\frac{2m}{n}-\beta]^{b}\}$, where $a+b=n-1$.
By Lemma \ref{lem-12}, we have $G \cong K_{\frac{2m}{n}+1}$ if $\alpha=\frac{2m}{n}$. Now we assume that $\alpha>\frac{2m}{n}$. In this assumption, we see that $G$ is not regular, since otherwise $q_1=\frac{2m}{n}+\alpha=\frac{4m}{n}$ by Lemma \ref{lem-4}, and then $\alpha=\frac{2m}{n}$, a contradiction. It remains to assume that $G$ is connected non-regular graph with $Q$-spectrum $Spec_{Q}(G)=\{\frac{2m}{n}+\alpha,[\frac{2m}{n}+\beta]^{a},[\frac{2m}{n}-\beta]^{b}\}$, where $a+b=n-1$, $\alpha>\frac{2m}{n}$ and $0<\beta\leq \frac{2m}{n}$.
Using (\ref{EQ-eq-12}), we have $\alpha=(n-1)\beta$, and in this situation,  $Spec_{Q}(G)=\{\frac{2m}{n}+(n-1)\beta,[\frac{2m}{n}+\beta]^{a},[\frac{2m}{n}-\beta]^{n-a-1}\}$. By Lemma \ref{lem-10}, we have $\frac{2m}{n}+(n-1)\beta+a(\frac{2m}{n}+\beta)+(n-a-1)(\frac{2m}{n}-\beta)=2m$, i.e., $2a\beta=0$. Since $\beta>0$, we have $a=0$. Therefore, $G$ has only two distinct $Q$-eigenvalues. By Lemma \ref{lem-14}, we have $G\cong K_{n}$, a contradiction. Summering above discussions, we know that $G \cong K_{\frac{2m}{n} +1}$ or $K_{\frac{2m}{n}+1,\frac{2m}{n}+1}\backslash F$ if $G$ is connected.

Next suppose that $G$ is disconnected.
If $G$ is a regular graph, then $\alpha=\frac{2m}{n}$ by Lemma \ref{lem-4}, which leads to $\beta=1$ as above. From (\ref{EQ-eq-13}), we have $Spec_{Q}(G)=\{[\frac{4m}{n}]^{s'},[0]^{s-s'},[\frac{2m}{n}+1]^{a},[\frac{2m}{n}-1]^{b}\}$. By Lemma \ref{lem-13}, we have $G \cong gK_{\frac{2m}{n} +1} \bigcup h(K_{\frac{2m}{n}+1,\frac{2m}{n}+1}\backslash F)$, where $\frac{2m}{n}\geq 2$ is an integer, $g=2s'-s$ and $h=s-s'$.
If $G$ is a non-regular graph, then $\alpha>\frac{2m}{n}$, since otherwise $q_1=\frac{2m}{n}+\alpha=\frac{4m}{n}$ and thus $\alpha=\frac{2m}{n}$, a contradiction.
Again from (\ref{EQ-eq-13}), we have $Spec_{Q}(G)=\{[\frac{2m}{n}+\alpha]^{s},[\frac{2m}{n}+\beta]^{a},[\frac{2m}{n}-\beta]^{b}\}$, where $a+b=n-s$, $\alpha>\frac{2m}{n}$ and $0<\beta\leq \frac{2m}{n}$.
Using (\ref{EQ-eq-12}), we have $\alpha=\frac{n-s}{s}\beta$. By Lemma \ref{lem-10}, we have $s(\frac{2m}{n}+\frac{n-s}{s}\beta)+a(\frac{2m}{n}+\beta)+(n-a-s)(\frac{2m}{n}-\beta)=2m$, i.e., $2a\beta=0$. Since $\beta>0$, we have $a=0$, which leads to
$Spec_{Q}(G)=\{[\frac{2m}{n}+\frac{n-s}{s}\beta]^{s},[\frac{2m}{n}-\beta]^{n-s}\}$ from  (\ref{EQ-eq-13}).
By Lemma \ref{lem-14},  $G$ is a union of some isomorphic complete graphs. It implies that $G$ is regular, a contradiction.

We complete this proof.
\end{proof}

In what the follows, we will simplify the lower bounds of  $QE(G)$ in Theorem \ref{thm-1} by estimating the parameter $\frac{\sqrt{\gamma_{1}\gamma_{n}}}{\gamma_{1}+\gamma_{n}}$.

\begin{cor}\label{cor-4}
Under the assumption of Theorem \ref{thm-1}, let $G$ be a connected graph with $n \geq 2$ vertices and $m \geq 1$ edges. If $\gamma_{n} \geq \frac{\sqrt{c}}{2n}$, where $c = m(n^{3}-n^{2}-2mn+4m)$, then
$$QE(G) \geq \frac{2\sqrt{2}}{3}\sqrt{[2m+\frac{1}{2}(\Delta-\delta)^{2}]n}$$
with equality holds if and only if $G \cong K_{3}$.
\end{cor}

\begin{proof}
By Lemma \ref{lem-7}, we have
$\gamma_{1}=|q_{1} - \frac{2m}{n}| = q_{1} - \frac{2m}{n} \leq \frac{2m +\sqrt{m(n^{3}-n^{2}-2mn+4m)}}{n} - \frac{2m}{n} = \frac{\sqrt{c}}{n}$.
Therefore, $\frac{\sqrt{c}}{2n} \leq \gamma_{n} \leq \gamma_{1} \leq \frac{\sqrt{c}}{n}$.
By Lemma \ref{lem-3}, we have $\frac{\sqrt{\gamma_{1}\gamma_{n}}}{\gamma_{1}+\gamma_{n}} \geq \frac{\sqrt{\frac{\sqrt{c}}{2n}\cdot\frac{\sqrt{c}}{n}}}{\frac{\sqrt{c}}{2n}+\frac{\sqrt{c}}{n}} = \frac{\sqrt{2}}{3}$.
By Theorem \ref{thm-1} and Lemma \ref{lem-8}, we have
\begin{equation}\\\\ \label{EQ-eq-14}
\begin{array}{ll}
QE(G)&\geq 2\sqrt{(2m+M_{1}-\frac{4m^{2}}{n})n} \cdot \frac{\sqrt{\gamma_{1}\gamma_{n}}}{\gamma_{1}+\gamma_{n}} \\
 &\geq \frac{2\sqrt{2}}{3}\sqrt{(2m+M_{1}-\frac{4m^{2}}{n})n}\\
 &\geq \frac{2\sqrt{2}}{3}\sqrt{[2m+\frac{1}{2}(\Delta-\delta)^{2}]n}
 \end{array}
\end{equation}
with the first equality holds if and only if $G \cong \frac{n}{2}K_{2}$ or $gK_{\frac{2m}{n} +1} \bigcup h(K_{\frac{2m}{n}+1,\frac{2m}{n}+1}\backslash F)$, where $\frac{2m}{n}\geq2$ is an integer,  the second equality holds if and only if $\frac{\sqrt{c}}{n} = \gamma_{1}$ and $\frac{\sqrt{c}}{2n} = \gamma_{n}$, and the last equality holds if and only if $G$ is a connected regular graph.

Now suppose that (\ref{EQ-eq-14}) is an equality. Then $G\cong K_{\frac{2m}{n}+1,\frac{2m}{n}+1}\backslash F$ or $ K_{\frac{2m}{n}+1}$ is a connected graph, and $2(\frac{2m}{n}+1)=n$, i.e., $m=\frac{n}{2}(\frac{n}{2}-1)$. Therefore,
$K_{\frac{2m}{n}+1,\frac{2m}{n}+1}= K_{\frac{n}{2},\frac{n}{2}}$. Similarly,
$K_{\frac{2m}{n} +1}= K_{n}$.
If $G \cong K_{\frac{n}{2},\frac{n}{2}}\backslash F$, we have $\frac{n}{2}-1=|q_{1} - \frac{2m}{n}| =\gamma_{1}=\frac{\sqrt{c}}{n}=\frac{\sqrt{\frac{1}{8}n^{5}-\frac{3}{2}n^{3}+n^{2}}}{n}$, where $c= m(n^{3}-n^{2}-2mn+4m)=\frac{1}{8}n^{5}-\frac{3}{2}n^{3}+n^{2}$. It implies that $n=1\pm\sqrt{5}$, a contradiction.
If $G \cong K_{n}$, we have $\gamma_{1}=\frac{\sqrt{c}}{n}=n-1$, where $c=[n(n-1)]^{2}$. On the other aspect, we have  $(n-1)-(n-2)=|q_{n} - \frac{2m}{n}|=\gamma_{n}=\frac{\sqrt{c}}{2n}=\frac{n-1}{2}$, which gives that $n=3$. Therefore, equality holds if and only if $G \cong K_{3}$.

Conversely, $Spec_{Q}(K_{3})=\{4,[1]^{2}\}$, and  thus $4=QE(K_{3}) = \frac{2\sqrt{2}}{3}\sqrt{[2m+\frac{1}{2}(\Delta-\delta)^{2}]n}=4$.
\end{proof}

\begin{cor}\label{cor-5}
Under the assumption of Theorem \ref{thm-1}, let $G$ be a connected graph with $n \geq 2$ vertices and $m \geq 1$ edges. If $\gamma_{n} \geq \frac{\sqrt{c}}{n^{3}}$, where $c = m(n^{3}-n^{2}-2mn+4m)$. Then
$$QE(G) > \frac{2n\sqrt{[2m+\frac{1}{2}(\Delta-\delta)^{2}]n}}{1+n^2}.$$
\end{cor}

\begin{proof}
By Lemma \ref{lem-7}, we have
$\gamma_{1}=|q_{1} - \frac{2m}{n}| = q_{1} - \frac{2m}{n} \leq \frac{2m +\sqrt{m(n^{3}-n^{2}-2mn+4m)}}{n} - \frac{2m}{n} = \frac{\sqrt{c}}{n}$.
Thus $\frac{\sqrt{c}}{n^{3}} \leq \gamma_{n} \leq \gamma_{1} \leq \frac{\sqrt{c}}{n}$.
By Lemma \ref{lem-3}, we have $\frac{\sqrt{\gamma_{1}\gamma_{n}}}{\gamma_{1}+\gamma_{n}} \geq \frac{\sqrt{\frac{\sqrt{c}}{n^{3}}\cdot\frac{\sqrt{c}}{n}}}{\frac{\sqrt{c}}{n^{3}}+\frac{\sqrt{c}}{n}} = \frac{n}{1+n^2}$.
By Theorem \ref{thm-1} and Lemma \ref{lem-8}, we have
\begin{equation}\\\\ \label{EQ-eq-15}
\begin{array}{ll}
QE(G)&\geq 2\sqrt{(2m+M_{1}-\frac{4m^{2}}{n})n} \cdot \frac{\sqrt{\gamma_{1}\gamma_{n}}}{\gamma_{1}+\gamma_{n}} \\
 &\geq 2\sqrt{(2m+M_{1}-\frac{4m^{2}}{n})n}\cdot\frac{n}{1+n^2}\\
 &\geq \frac{2n\sqrt{[2m+\frac{1}{2}(\Delta-\delta)^{2}]n}}{1+n^2}.
 \end{array}
\end{equation}

Additionally, as the arguments as the Corollary \ref{cor-4}, the equality (\ref{EQ-eq-15}) holds if and only if $G\cong K_{n}$ and $1=(n-1)-(n-2)=|q_{n} - \frac{2m}{n}|=\gamma_{n}=\frac{\sqrt{c}}{n^{3}}=\frac{n-1}{n^2}$, which gives $n^2-n+1=0$, a contradiction. Thus the equality can not achieve.
\end{proof}

\begin{remark}\label{rem-2}
The lower bound described in Corollary \ref{cor-5} depend on the assumption of $\gamma_{n} \geq \frac{\sqrt{c}}{n^{3}}$. In fact, there exist a great large of graphs satisfying the algebraic condition $\gamma_{n} \geq \frac{\sqrt{c}}{n^{3}}$. It is easy to see that $\lim_{n\to \infty} \frac{\sqrt{c}}{n^{3}}=0$. We ask if there exists a sufficiently small number $\varepsilon>0$ such that $\gamma_n\ge\varepsilon$ for any $n$. It is an interesting problem to characterize such  graphs satisfying $\gamma_n\ge\varepsilon$. However, on the other aspect, by setting $\gamma_{n}= 0$ we can also  improve the lower bond of $QE(G)$, which is presented  in Theorem \ref{thm-2}.
\end{remark}

\begin{thm}\label{thm-2}
Let $G$ be a connected graph with $n$ vertices and $m$ edges. Assume that $\gamma_{n}= 0$. Then
\begin{equation}\label{EQ-eq-16}\\\\ QE(G) \geq \frac{2m+ M_{1}-\frac{4m^{2}}{n}}{\gamma_1}\end{equation}
with equality holds if and only if $G\cong K_{\frac{n}{2},\frac{n}{2}}$.
\end{thm}

\begin{proof}
By Lemma \ref{lem-2}, we have $QE(G)=\sum_{i=1}^{n}\gamma_{i}\geq \frac{\sum_{i=1}^{n}\gamma_{i}^{2} + n\gamma_{1}\gamma_{n}}{\gamma_{1} + \gamma_{n}}$. Using (\ref{EQ-eq-9}) and $\gamma_{n}= 0$, we have $QE(G)\geq \frac{2m+M_{1}-\frac{4m^{2}}{n}}{\gamma_{1}}$ with
the equality holds if and only if $\gamma_{1} = \cdots = \gamma_{s}$ and $\gamma_{s+1} = \cdots = \gamma_{n}=0$ for some $1\le s< n$.

Now we suppose that (\ref{EQ-eq-16}) is an equality. Then there exists some $1\le s< n$ such that $\gamma_{i}=|q_i-\frac{2m}{n}|=\alpha$ for $i=1,2,...,s$ and $\gamma_{j}=|q_j-\frac{2m}{n}|=0$ for $j=s+1,...,n$,  we have $q_{i}\in \{\alpha+\frac{2m}{n},-\alpha+\frac{2m}{n}\}$ for $i=1,...,s$ and $q_{j}=\frac{2m}{n}$ for $j=s+1,...,n$.
According to Remark \ref{rem-1}, we claim that $q_1=\alpha+\frac{2m}{n}$ is $Q$-spectral radius of $G$.
Since $G$ is connected graph, then $G$ has $Q$-spectrum:
\begin{equation}\label{EQ-eq-17}
\\\\Spec_{Q}(G)=\{\frac{2m}{n}+\alpha,[\frac{2m}{n}-\alpha]^{s-1},[\frac{2m}{n}]^{n-s}\}.
\end{equation}
Since $q_1\ge \frac{4m}{n}$ by Lemma \ref{lem-4}, we have $\alpha\ge \frac{2m}{n}$.

First suppose that $G$ is connected $r$-regular graph, we have $\alpha= \frac{2m}{n}=r$ by Lemma \ref{lem-4}. Therefore, $Spec_{Q}(G)=\{2r,[r]^{n-s},[0]^{s-1}\}$. By Lemma \ref{lem-10}, we have $2r+(n-s)r=2m=nr$ and $(2r)^{2}+(n-s)r^{2}=2m+M_{1}=nr+nr^{2}$,
 i.e., $s=2$ and $r=\frac{n}{2}$, which leads to $Spec_{Q}(G)=\{n,[\frac{n}{2}]^{n-2},0\}$.
By Lemma \ref{lem-9}, the multiplicity of $Q$-eigenvalue $0$ equals the number of even components of $G$. It implies that $G$ is a connected bipartite $\frac{n}{2}$-regular graph, we have $G\cong K_{\frac{n}{2},\frac{n}{2}}$.

Next suppose that $G$ is not regular graph, we have $\alpha> \frac{2m}{n}$ by Lemma \ref{lem-4}. Thus, if $s>1$, then $q_{i}=\frac{2m}{n}-\alpha<0$ and it contradicts $q_{i} \geq 0$. Therefore, $s=1$, which leads to $Spec_{Q}(G)=\{\frac{2m}{n}+\alpha,[\frac{2m}{n}]^{n-1}\}$ form (\ref{EQ-eq-17}). It implies that $G$ has only two distinct $Q$-eigenvalues. By Lemma \ref{lem-14}, we have $G\cong K_{n}$, a contradiction.

We complete this proof.
\end{proof}

\begin{cor}\label{cor-2}
Let $G$ be a connected graph with $n$ vertices and $m$ edges. Assume that $\gamma_{n}= 0$. Then
$$QE(G) \left\{\begin{array}{ll}> \frac{2m+\frac{1}{2}(\Delta-\delta)^{2}}{2\Delta-\frac{2m}{n}}& \mbox{ if $G$ is not regular},\\
\geq n& \mbox{ if $G$ is regular with equality iff  $G\cong K_{\frac{n}{2},\frac{n}{2}}$.}
\end{array}\right.
$$
\end{cor}

\begin{proof}
By Lemma \ref{lem-18}, we have $\gamma_1=q_1-\frac{2m}{n}\leq 2\Delta-\frac{2m}{n}$. By Theorem \ref{thm-2} and Lemma \ref{lem-8}, we have
$$
\begin{array}{ll}
QE(G)
& \geq \frac{2m+\frac{1}{2}(\Delta-\delta)^{2}}{\gamma_1}  \geq \frac{2m+\frac{1}{2}(\Delta-\delta)^{2}}{2\Delta-\frac{2m}{n}}
\end{array}
$$
with the first equality holds if and only if $G\cong K_{\frac{n}{2},\frac{n}{2}}$, and the second equality holds if and only if $G$ is regular by Lemma \ref{lem-18}. Hence, if $G$ is not regular, then $QE(G)  > \frac{2m+\frac{1}{2}(\Delta-\delta)^{2}}{2\Delta-\frac{2m}{n}}$, and if $G$ is regular, then $QE(G)\geq\frac{2m+\frac{1}{2}(\Delta-\delta)^{2}}{2\Delta-\frac{2m}{n}}=n$, in this situation, $G\cong K_{\frac{n}{2},\frac{n}{2}}$.
\end{proof}

If $G$ is a regular graph, then $2m=nr$, $M_{1} = nr^{2}$ and $\gamma_{1}=r$. Using Theorem \ref{thm-1} and Corollary \ref{cor-2}, we directly get the lower bound for $QE(G)$ of regular graph $G$, which can also be viewed as the bound of $E(G)$.

\begin{cor}\label{cor-3}
Let $G$ be a connected $r$-regular graph with $n$ vertices and $m$ edges. Then
$$
E(G)=QE(G) \geq\left\{\begin{array}{ll}
{n} & \mbox{ if $\gamma_{n} = 0$ with equality iff $G\cong K_{\frac{n}{2},\frac{n}{2}}$},\\
{2nr \cdot \frac{\sqrt{\gamma_{n}}}{r+\gamma_{n}}} & \mbox{ if $\gamma_{n} > 0$ with equality iff $G \cong K_{n}$ or $ K_{\frac{n}{2},\frac{n}{2}}\backslash F$},
\end{array}\right.
$$
where $F$ is perfect matching of the bipartite graph $K_{\frac{n}{2},\frac{n}{2}}$.
\end{cor}

A $r$-regular graph $G$ on $n$ vertices is called \emph{strongly regular graph} with parameters $(n,r,a,c)$ if any two adjacent vertices has $a \geq 0$ common neighbours and any two non-adjacent vertices has $c \geq 0$ common neighbours. In particular,
the strongly regular graph with parameters $(n,r,\frac{r(r-1)}{n-1},\frac{r(r-1)}{n-1})$ is denoted by $S(n,r)$. Such a strongly regular graph exists, one can refer to {\cite{Axler}} for more details.
For example, by taking $n=(t+3)(t+1)^{2}$ and $r=(t+2)(t+1)$, we have $\frac{r(r-1)}{n-1}=t+1$  and $S(n,r)$ will be the strongly regular graph with parameters $((t+3)(t+1)^{2},(t+2)(t+1),t+1,t+1)$, which is the so called  point graph of generalized quadrangle with order $(t+2,t)$ (see Lemma 10.8.1 of {\cite{Axler}}).
In the following, we will give the upper bound of $QE(G)$,  which is achieved by $S(n,r)$.

\begin{thm}\label{thm-3}
Let $G$ be a graph with $n$ vertices and $m$ edges. Then
$$\scriptsize{QE(G) \left\{\begin{array}{ll}
\!\!\!\leq \!\frac{2m}{n} + \sqrt{(n-1)[2m + M_{1} - \frac{4m^{2}}{n}-(\frac{2m}{n})^{2}]}&\mbox{\!\!\!\!\!if $n\!\leq\!\frac{8m^{2}}{2m+M_{1}}$ with equality iff $G \cong K_{n}$, $\frac{n}{2}K_{2}$ or $G\cong S(n,r)$},\\
\!\!\!<\!\sqrt{\frac{2m + M_{1} - \frac{4m^{2}}{n}}{n}} + \sqrt{(n-1)(2m + M_{1} - \frac{4m^{2}}{n}-\frac{2m + M_{1} - \frac{4m^{2}}{n}}{n})}&\mbox{\!\!\!\!\!if $n\!>\!\frac{8m^{2}}{2m+M_{1}}$.}
\end{array}\right.}$$
\end{thm}

\begin{proof}
Let $\eta_i=|q_{i} - \frac{2m}{n}|$ be ordered as $\eta_1\ge\eta_2\ge\cdots\ge\eta_n$. According to (\ref{EQ-eq-9}) and Cauchy-Schwartz inequality, we have
\begin{equation}\\\\\label{EQ-eq-29}\begin{array}{ll}
QE(G)&=|\eta_{1}| + \sum_{i=2}^{n}|\eta_{i}|\\
&\leq \eta_{1} + \sqrt{(n-1)\sum_{i=2}^{n}\eta_{i}^{2}}\\
&=\eta_{1} + \sqrt{(n-1)(\sum_{i=1}^{n}\eta_{i}^{2}-\eta_{1}^{2})}\\
&= \eta_{1} + \sqrt{(n-1)(2m + M_{1} - \frac{4m^{2}}{n} - \eta_{1}^{2})}.
\end{array}
\end{equation}

Now we consider the function
$f(x) = x + \sqrt{(n-1)(2m + M_{1} - \frac{4m^{2}}{n}- x^{2})}$, where $ 0 \leq x \leq \sqrt{2m + M_{1} - \frac{4m^{2}}{n}}$ is a  variable standing for $\eta_1$.
Note that $f'(x) = 1 + \sqrt{n-1} \cdot \frac{-x}{\sqrt{2m + M_{1} - \frac{4m^{2}}{n} - x^{2}}}$, we see that $f(x)$ decreases on $U_1\!=\!\{x\mid \sqrt{\frac{2m + M_{1} - \frac{4m^{2}}{n}}{n}} \!\le\! x \!\leq \!\sqrt{2m + M_{1} - \frac{4m^{2}}{n}}\}$ and increases on $U_2\!=\!\{x\mid 0 \leq x \leq \sqrt{\frac{2m + M_{1} - \frac{4m^{2}}{n}}{n}}\}$.
Since $\eta_{1} = q_{1} - \frac{2m}{n} \geq \frac{2m}{n}$ by Lemma \ref{lem-4}, we see that
$$
f(\eta_{1}) \leq \left\{\begin{array}{ll}f(\frac{2m}{n})& \mbox{ if $\frac{2m}{n}\in U_1$,}\\
 f(\sqrt{\frac{2m + M_{1} - \frac{4m^{2}}{n}}{n}})& \mbox{ if $\frac{2m}{n}\in U_2$.}
\end{array}\right.
$$

{\flushleft\bf Case 1.} Suppose that $\frac{2m}{n}\in U_1$ ( equivalently $\frac{2m}{n} \geq \sqrt{\frac{2m + M_{1} - \frac{4m^{2}}{n}}{n}}\Longleftrightarrow n\leq\frac{8m^{2}}{2m+M_{1}}$).

In this case, we have
\begin{equation}\\\\\label{EQ-eq-19}\begin{array}{ll}
QE(G)\leq f(\eta_1)&
\leq f(\frac{2m}{n})\\
& =\frac{2m}{n} + \sqrt{(n-1)[2m + M_{1} - \frac{4m^{2}}{n}-(\frac{2m}{n})^{2}]}.
\end{array}
\end{equation}
The first equality of (\ref{EQ-eq-19}) holds if and only if $\eta_{2} = \eta_{3} = \cdots = \eta_{n}$ and the second equality holds if and only if
$\eta_{1} = \frac{2m}{n}$, i.e., $q_{1} = \frac{4m}{n}$, which implies that $G$ is $r=\frac{2m}{n}$ regular by Lemma \ref{lem-4}.

Now suppose that (\ref{EQ-eq-19}) is an equality. Then $G$ is a $r=\frac{2m}{n}$ regular graph and   $\eta_{i}=|q_{i}-\frac{2m}{n}|= \sqrt{\frac{2m + M_{1} - \frac{4m^{2}}{n} - \eta_{1}^{2}}{n-1}}=\sqrt{\frac{r(n-r)}{n-1}}$ for $i=2,3,...,n$. Thus we have
\begin{equation}\\\\\label{EQ-eq-27}\{q_2,q_3,...,q_n\}\subseteq \{\sqrt{\frac{r(n-r)}{n-1}}+r,-\sqrt{\frac{r(n-r)}{n-1}}+r\}\ \mbox{ and $q_1=2r$ }.\end{equation}

First suppose that $G$ is connected.
From (\ref{EQ-eq-27}), the $Q$-spectrum of $G$ has three choices: $Spec_{Q}(G)=\{2r,[\sqrt{\frac{r(n-r)}{n-1}}+r]^{n-1}\}$, $Spec_{Q}(G)=\{2r,[-\sqrt{\frac{r(n-r)}{n-1}}+r]^{n-1}\}$ or $Spec_{Q}(G)=\{2r,[\sqrt{\frac{r(n-r)}{n-1}}+r]^{b},[-\sqrt{\frac{r(n-r)}{n-1}}+r]^{n-b-1}\}$.
If $Spec_{Q}(G)=\{2r,[\sqrt{\frac{r(n-r)}{n-1}}+r]^{n-1}\}$, then, by Lemma \ref{lem-10}, we have
$$\left\{\begin{array}{ll}
2r+(\sqrt{\frac{r(n-r)}{n-1}}+r)(n-1)=2m=nr,\\
(2r)^{2}+(\sqrt{\frac{r(n-r)}{n-1}}+r)^{2}(n-1)=2m+M_{1}=nr+nr^{2}.
\end{array}\right.
$$It follows that  $\sqrt{\frac{r(n-r)}{n-1}}=r-n<0$, a contradiction.
If $Spec_{Q}(G)=\{2r,[-\sqrt{\frac{r(n-r)}{n-1}}+r]^{n-1}\}$, then, as similar as above, we get $\sqrt{\frac{r(n-r)}{n-1}}=n-r$,  which leads to $Spec_{Q}(G)=\{2(n-1),[n-2]^{n-1}\}$, and so  $G\cong K_{n}$ by Lemma \ref{lem-14}.
If $Spec_{Q}(G)=\{2r,[\sqrt{\frac{r(n-r)}{n-1}}+r]^{b},[-\sqrt{\frac{r(n-r)}{n-1}}+r]^{n-b-1}\}$, then $G$ has three $A$-eigenvalues due to $G$ is regular. By Lemma \ref{lem-17}, $G$ is a strongly regular graph with parameters $(n,r,a,c)$, and $Spec_{A}(G)=\{r,[\sqrt{\frac{r(n-r)}{n-1}}]^{b},[-\sqrt{\frac{r(n-r)}{n-1}}]^{n-b-1}\}$. It is well known that
the $A$-eigenvalues \!$\sqrt{\frac{r(n-r)}{n-1}}$ and $-\!\!\sqrt{\frac{r(n-r)}{n-1}}$ of $G$ satisfy the equation $x^{2}-(a-c)x-(r-c)=0$, and $r+b\sqrt{\frac{r(n-r)}{n-1}}+(n-b-1)(-\sqrt{\frac{r(n-r)}{n-1}})=0$.
By simple calculation, we have $a=c=\frac{r(r-1)}{n-1}$ and $b=\frac{(n-1)\sqrt{r-c}-r}{2\sqrt{r-c}}$. It follows that $G$ is a strongly regular graph with parameters $(n,r,\frac{r(r-1)}{n-1},\frac{r(r-1)}{n-1})$, and thus $G \cong S(n,r)$.

Next suppose that $G$ is disconnected. From (\ref{EQ-eq-27}),  there exists some $2\leq b< n$ such that  $Spec_{Q}(G)=\{[2r]^{b+1},[-\sqrt{\frac{r(n-r)}{n-1}}+r]^{n-b-1}\}$, which implies that $\sqrt{\frac{r(n-r)}{n-1}}=r$, i.e., $r=1$.
By Lemma \ref{lem-10}, we have $2(b+1)=2m=n$, i.e., $b=\frac{n}{2}-1$. Therefore, $Spec_{Q}(G)=\{[2]^{\frac{n}{2}},[0]^{\frac{n}{2}}\}$. It follows that $G \cong \frac{n}{2}K_{2}$.

{\flushleft\bf Case 2.} Suppose that $\frac{2m}{n}\in U_2$ ( equivalently $\frac{2m}{n} < \sqrt{\frac{2m + M_{1} - \frac{4m^{2}}{n}}{n}}\Longleftrightarrow n>\frac{8m^{2}}{2m+M_{1}}$).

In this case, we have
\begin{equation}\\\label{EQ-eq-20}\begin{array}{ll}
QE(G) \leq f(\eta_1)
&\leq f(\sqrt{\frac{2m + M_{1} - \frac{4m^{2}}{n}}{n}})\\
&=\sqrt{\frac{2m + M_{1} - \frac{4m^{2}}{n}}{n}} + \sqrt{(n-1)(2m + M_{1} - \frac{4m^{2}}{n}-\frac{2m + M_{1} - \frac{4m^{2}}{n}}{n})}.
\end{array}
\end{equation}
The first equality of (\ref{EQ-eq-20}) holds if and only if $\eta_{2} = \eta_{3} = \cdots = \eta_{n}$ and the second equality holds if and only if
$\eta_{1} = \sqrt{\frac{2m + M_{1} - \frac{4m^{2}}{n}}{n}}$, i.e., $q_{1}=\sqrt{\frac{2m + M_{1} - \frac{4m^{2}}{n}}{n}}+\frac{2m}{n}$.

Now suppose that (\ref{EQ-eq-20}) is an equality. Then $\eta_{1} = \sqrt{\frac{2m + M_{1} - \frac{4m^{2}}{n}}{n}}$ and $\eta_{i}=|q_{i}-\frac{2m}{n}|= \sqrt{\frac{2m + M_{1} - \frac{4m^{2}}{n} - \eta_{1}^{2}}{n-1}}=\sqrt{\frac{2m + M_{1} - \frac{4m^{2}}{n}}{n}}$  for $i=2,3,...,n$. We have

\begin{equation}\\\\\label{EQ-eq-28}
\tiny{\{q_2,q_3,...,q_n\}\subseteq \{\sqrt{\frac{2m + M_{1} - \frac{4m^{2}}{n}}{n}}+\frac{2m}{n},-\sqrt{\frac{2m + M_{1} - \frac{4m^{2}}{n}}{n}}+\frac{2m}{n}\}\ and\ q_{1}=\sqrt{\frac{2m + M_{1} - \frac{4m^{2}}{n}}{n}}+\frac{2m}{n}.}
\end{equation}

First suppose that $G$ is connected. From (\ref{EQ-eq-28}), $Spec_{Q}\!=\!\{\! \sqrt{\frac{2m + M_{1} - \frac{4m^{2}}{n}}{n}}+\frac{2m}{n}\!, \![-\!\sqrt{\frac{2m + M_{1} - \frac{4m^{2}}{n}}{n}}+\frac{2m}{n}]^{n-1}\}$.
Since $\frac{2m}{n} < \sqrt{\frac{2m + M_{1} - \frac{4m^{2}}{n}}{n}}$, we have $q_{i}=-\sqrt{\frac{2m + M_{1} - \frac{4m^{2}}{n}}{n}}+\frac{2m}{n}<0$ for $i=2,3,...,n$, and it contradicts $q_i\geq 0$.

Next suppose that $G$ is disconnected. From (\ref{EQ-eq-28}), there exists some $2\leq b< n$ such that $Spec_{Q}(G)=\{[\sqrt{\frac{2m + M_{1} - \frac{4m^{2}}{n}}{n}}+\frac{2m}{n}]^{b+1},[-\sqrt{\frac{2m + M_{1} - \frac{4m^{2}}{n}}{n}}+\frac{2m}{n}]^{n-b-1}\}$. This  is also impossible as  above.

We complete this proof.
\end{proof}

In the following, by applying Lemma \ref{lem-19}, we can simplify the upper bounds of $QE(G)$ in Theorem \ref{thm-3}.

\begin{cor}\label{cor-6}
Let $G$ be a connected nonregular graph with $n$ vertices and $m$ edges. Then
$$\small{QE(G)< \left\{\begin{array}{ll}
 \frac{2m}{n} + \sqrt{(n-1)[2m + \frac{n}{4}(\Delta-\delta)^{2}-(\frac{2m}{n})^{2}]}& \mbox{ if $n\leq\frac{4m(\sqrt{1+(\Delta-\delta)^{2}}-1)}{(\Delta-\delta)^{2}}$},\\
 \sqrt{\frac{2m}{n}+\frac{1}{4}(\Delta-\delta)^{2}} + \sqrt{(n-1)(2m +\frac{n-1}{4}(\Delta-\delta)^{2}-\frac{2m}{n})} & \mbox{ if $n>\frac{4m(\sqrt{1+(\Delta-\delta)^{2}}-1)}{(\Delta-\delta)^{2}}$.}
\end{array}\right.}$$
\end{cor}

\begin{proof}
By (\ref{EQ-eq-29}), we have $QE(G)\leq\eta_{1} + \sqrt{(n-1)(2m + M_{1} - \frac{4m^{2}}{n} - \eta_{1}^{2})}$, from which, by substituting $M_1$ according to   Lemma \ref{lem-19}, we get
$$QE(G)\le\eta_{1} + \sqrt{(n-1)(2m + \frac{n}{4}(\Delta-\delta)^2 - \eta_{1}^{2})}.$$
The above inequality must be strict since $G$ is a connected nonregular graph.
Now, we define
$g(x) = x + \sqrt{(n-1)(2m + \frac{n}{4}(\Delta-\delta)^2- x^{2})}$, where $ 0 \leq x \leq \sqrt{2m + \frac{n}{4}(\Delta-\delta)^2}$ is a variable standing for $\eta_1$.
As  similar as $f(x)$ in the proof of Theorem \ref{thm-3},  $g(x)$ decreases on $I_1=\{x\mid \sqrt{\frac{2m}{n}+\frac{1}{4}(\Delta-\delta)^2} \le x \leq \sqrt{2m + \frac{n}{4}(\Delta-\delta)^2}\}$ and increases on $I_2=\{x\mid 0 \leq x \leq \sqrt{\frac{2m}{n}+\frac{1}{4}(\Delta-\delta)^2}\}$.
Since $\eta_{1} = q_{1} - \frac{2m}{n} > \frac{2m}{n}$ by Lemma \ref{lem-4}, we have
$$
QE(G)<g(\eta_{1}) \leq \left\{\begin{array}{ll}g(\frac{2m}{n})& \mbox{ if $\frac{2m}{n}\in I_1$},\\
 g(\sqrt{\frac{2m + \frac{n}{4}(\Delta-\delta)^2}{n}})& \mbox{ if $\frac{2m}{n}\in I_2$,}
\end{array}\right.
$$
which is just what we need because of $\frac{2m}{n}\in I_1$ if and only if $n\leq\frac{4m(\sqrt{1+(\Delta-\delta)^{2}}-1)}{(\Delta-\delta)^{2}}$, and $\frac{2m}{n}\in I_2$ if and only if $n>\frac{4m(\sqrt{1+(\Delta-\delta)^{2}}-1)}{(\Delta-\delta)^{2}}$.
\end{proof}

If $G$ is a regular graph, then $M_{1} = nr^{2}$, $2m=nr$ and so $n\leq\frac{8m^{2}}{2m+M_{1}}$.
Using Theorem \ref{thm-3}, we directly get the upper bound for $QE(G)$ of regular graph $G$, which can also be viewed as the bounds of $E(G)$.

\begin{cor}\label{cor-7}
Let $G$ be a regular graph with $n$ vertices and $m$ edges. Then
$$
E(G)=QE(G) \leq \frac{2m}{n} + \sqrt{(n-1)[2m - (\frac{2m}{n})^{2}]}
$$
with equality holds if and only if $G \cong K_{n}$, $\frac{n}{2}K_{2}$ or $G\cong S(n,r)$.
\end{cor}

\begin{remark}\label{rem-3}
For a regular graph, the result of Corollary \ref{cor-7} is the same as that of Theorem 1 mentioned by J.H. Koolen et al. in {\cite{Koolen}}.
\end{remark}

\section{Applications}
By applying previous Theorems and Corollaries, in this section, we  give two  bounds of $QE(G)$ for specific regular graphs, from which we list  two tables that compare the values of  the various bounds of $QE(G)$ and it  indicates the improvement of our bounds for $QE(G)$.

$(3, 6)$-Fullerene arise in chemistry as molecules consisting entirely of carbon atoms. Each carbon atom is bonded to exactly three others, thus the vertices of the graph represent the carbon atoms, the edges the bonded pairs of atoms, and the $(3, 6)$-fullerene is a connected $3$-regular graph  with all faces $3$-cycles or $6$-cycles.
Since fullerene has always been an important research object in the fields of mathematics, physics and chemistry due to its unique structure and excellent physical properties, Matt DeVos et al. in {\cite{DeVos}} give the spectrum of $(3,6)$-fullerene has the form $\{3,-1,-1,-1\}\bigcup L \bigcup(-L)$, where $L$ is a multiset of nonnegative real numbers, and $-L$ is the multiset of their negatives. Here  we estimate the bounds for the energy of $(3,6)$-fullerene.

Since $(3,6)$-fullerene has the spectrum of the form $\{3,-1,-1,-1\}\bigcup L \bigcup(-L)$, from Corollary \ref{cor-3} and Corollary \ref{cor-7}, one can directly get the bounds of the energy of $(3,6)$-fullerene in the following result.

\begin{thm}\label{exa-1}
Let $G$ be a $(3,6)$-fullerene with $n$ vertices. Then
$$3+\sqrt{3(n-1)(n-3)}>
E(G)=QE(G) >\left\{\begin{array}{ll}
n & \mbox{ if\  $|\gamma_{n}|=0$}, \\
6n \cdot \frac{\sqrt{|\gamma_{n}|}}{3+|\gamma_{n}|} & \mbox{ if\  $0 < |\gamma_{n}| < 1$},\\
\frac{3n}{2} &  \mbox{ if\  $|\gamma_{n}| \geq1$}.
\end{array}\right.
$$
\end{thm}

\begin{remark}\label{rem-4}
The limit of ratio of the upper and lower bound of the energy of $(3,6)$-fullerene is
$$\left\{\begin{array}{ll}\lim_{n\to \infty} \frac{3+\sqrt{3(n-1)(n-3)}}{n}=\sqrt{3}& \mbox{ if $|\gamma_{n}|=0$},\\
  \lim_{n\to \infty} \frac{3+\sqrt{3(n-1)(n-3)}}{\frac{3n}{2}}=\frac{2\sqrt{3}}{3}& \mbox{ if $|\gamma_{n}| \geq1$.}
  \end{array}\right.
$$
It implies that while $n$ is sufficiently large, we have
$$\left\{\begin{array}{ll}QE(G)\in (n, \sqrt{3}n]& \mbox{ if $|\gamma_{n}|=0$},\\
  QE(G)\in (\frac{3}{2}n, \frac{2\sqrt{3}}{3}n]& \mbox{ if $|\gamma_{n}| \geq1$}.
  \end{array}\right.
$$
We are not sure wether there exist infinite $(3,6)$-fullerene such that the corresponding $\gamma_n$ equals zero, tends to zero or great than some positive constant. However, there indeed exist such regular graphs. One can refer to Remark \ref{rem-5} for details.
\end{remark}

The \emph{cartesian product} of simple graphs $G_1$ and $G_2$ is denoted by $G_1\Box G_2$. In particular, for $n\geq3$, the cartesian product $C_n\Box P_2$ is a polyhedral graph and called the \emph{$n$-prism}. One can refer to Section 1.4 of {\cite{Bondy}} for more details.
Now we use Corollary \ref{cor-3} and Corollary \ref{cor-7} to give the bound for the $QE(G)$ of $n$-prism. First of all  we need a lemma bellow.

\begin{lem}\label{lem-20}
Let $G=C_n\Box P_2$ be a $n$-prism with $2n\geq6$ vertices. Then
$$\gamma_{2n}(G)=\left\{\begin{array}{ll}
{2\cos(\frac{2\pi\lfloor\frac{n}{6}\rfloor}{n})-1} & \mbox{if\  $n=6k+1$ or $6k+2$},\\
{1-2\cos(\frac{2\pi\lceil\frac{n}{6}\rceil}{n})} & \mbox{if\ $n=6k+4$ or $6k+5$}.
\end{array}\right.$$
\end{lem}
\begin{proof}
It is well known that the $Spec_{Q}(C_n)=\{\ 2\cos{(\frac{2\pi j}{n})}+2 \mid  j=0,1,...,n-1\ \}$ and $Spec_{Q}(P_2)=\{\ 2,0\ \}$. By Lemma \ref{lem-16}, we have
$Spec_{Q}(G)=\{\ 2\cos{(\frac{2\pi j}{n})}+4,\ 2\cos{(\frac{2\pi j}{n})}+2 \ | \ j=0,1,...,n-1\ \}$. Thus $$\gamma_{2n}(G)=\min_{0\le j\le n-1}\{\ | 2\cos{(\frac{2\pi j}{n})}+1 |,\  |2\cos{(\frac{2\pi j}{n})}-1  |\  \}.$$ It is clear that $\gamma_{2n}(G)=0$ iff
$ 2\cos{(\frac{2\pi j}{n})}+1 =0$ or $2\cos{(\frac{2\pi j}{n})}-1=0$ iff $j=\frac{n}{3}$, $\frac{2n}{3}$ or $j=\frac{n}{6}$, $\frac{5n}{6}$ iff $n=0$ ($\mod3$) or $n=0$ ($\mod6$), i.e., $\gamma_{2n}(G)=0$ iff $n=0$ ($\mod3$). Therefore, $\gamma_{2n}>0$ iff $n\not=0$ ($\mod3$). In what follows we suppose that $n=6k+1$, $6k+2$, $6k+4$ or $6k+5$.

Let $\phi(j)=| 2\cos{(\frac{2\pi j}{n})}+1 |$ and $\varphi(j)=| 2\cos{(\frac{2\pi j}{n})}-1 |$, where $0\le j<n$. By putting $a=\min_{0\le j\le n-1} \phi(j)$ and $b=\min_{0\le j\le n-1}\varphi(j)$, we have  $\gamma_{2n}(C_n\Box P_2)=\min\{a,b\}$. First we determine $a$.
By considering the monotonicity of $\cos(x)$ on $(\frac{\pi}{2},\pi)$, we see that $\phi(j)=2 |\cos{(\frac{2\pi j}{n})}-\cos(\frac{2\pi}{3})|$ achieves its minimum value at
$0\leq j_1\leq n-1$ such that  $|\frac{2\pi j_1}{n}-\frac{2\pi}{3}|=\frac{|3j_1-n|2\pi}{3n}$ is as small as possible. It implies that
$$|\frac{2\pi j_1}{n}-\frac{2\pi}{3}|=\frac{|3j_1-n|2\pi}{3n}=\left\{\begin{array}{ll}
\frac{2\pi}{3n}& \mbox{ we take $j_1=k=\lfloor\frac{n}{3}\rfloor$ if $n=3k+1$},\\
\frac{2\pi}{3n}& \mbox{ we take $j_1=k+1=\lceil\frac{n}{3}\rceil$ if $n=3k+2$}.\\
\end{array}\right.
$$
Therefore,
$$\footnotesize{a=\left\{\begin{array}{ll}
\phi(\lfloor\frac{n}{3}\rfloor)=2 |\cos(\frac{2\pi \lfloor\frac{n}{3}\rfloor}{n})-\cos(\frac{2\pi}{3})|=2|\cos(\frac{2\pi }{3+\frac{1}{k}})-\cos(\frac{2\pi}{3})|=2\cos(\frac{2\pi }{3+\frac{1}{k}})+1 & \mbox{ if $n=3k+1$},\\
\phi(\lceil\frac{n}{3}\rceil)=2 |\cos{(\frac{2\pi \lfloor\frac{n}{3}\rfloor}{n})}-\cos(\frac{2\pi}{3})|=2|\cos(\frac{2\pi }{3-\frac{1}{k+1}})-\cos(\frac{2\pi}{3})|=-2\cos(\frac{2\pi }{3-\frac{1}{k+1}})-1 & \mbox{ if $n=3k+2$}.\\
\end{array}\right.}
$$
Similarly, $\varphi(j)=2 |\cos{(\frac{2\pi j}{n})}-\cos(\frac{\pi}{3})|$
achieves its minimum value at
$0\leq j_2\leq n-1$ such that  $|\frac{2\pi j_2}{n}-\frac{\pi}{3}|=\frac{|6j_2-n|\pi}{3n}$ is as small as possible. It implies  that
$$|\frac{2\pi j_2}{n}-\frac{\pi}{3}|=\frac{|6j_2-n|\pi}{3n}=\left\{\begin{array}{ll}
\frac{\pi}{3n}& \mbox{ we take $j_2=k=\lfloor\frac{n}{6}\rfloor$ if $n=6k+1$},\\
\frac{2\pi}{3n}& \mbox{ we take $j_2=k=\lfloor\frac{n}{6}\rfloor$ if $n=6k+2$},\\
\frac{2\pi}{3n}& \mbox{ we take $j_2=k+1=\lceil\frac{n}{6}\rceil$ if $n=6k+4$},\\
\frac{\pi}{3n}& \mbox{ we take $j_2=k+1=\lceil\frac{n}{6}\rceil$ if $n=6k+5$}.
\end{array}\right.
$$
Therefore,
$$\footnotesize{b=\left\{\begin{array}{ll}
\varphi(\lfloor\frac{n}{6}\rfloor)=2 |\cos{(\frac{2\pi\lfloor\frac{n}{6}\rfloor}{n})}-\cos{(\frac{\pi}{3})}|=2|\cos{(\frac{2\pi}{6+\frac{1}{k}})}- \cos{(\frac{\pi}{3})}|=2\cos{(\frac{2\pi}{6+\frac{1}{k}})}-1 & \mbox{  if $n=6k+1$},\\
\varphi(\lfloor\frac{n}{6}\rfloor)=2 |\cos{(\frac{2\pi\lfloor\frac{n}{6}\rfloor}{n})}-\cos{(\frac{\pi}{3})}|=2|\cos{(\frac{2\pi}{6+\frac{2}{k}})}- \cos{(\frac{\pi}{3})}|=2\cos{(\frac{2\pi}{6+\frac{2}{k}})}-1 & \mbox{  if  $n=6k+2$},\\
\varphi(\lceil\frac{n}{6}\rceil)=2 |\cos{(\frac{2\pi\lceil\frac{n}{6}\rceil}{n})}-\cos{(\frac{\pi}{3})}|=2|\cos{(\frac{2\pi}{6-\frac{2}{k+1}})}- \cos{(\frac{\pi}{3})}|=1-2\cos{(\frac{2\pi}{6-\frac{2}{k+1}})} & \mbox{ if $n=6k+4$},\\
\varphi(\lceil\frac{n}{6}\rceil)=2 |\cos{(\frac{2\pi\lceil\frac{n}{6}\rceil}{n})}-\cos{(\frac{\pi}{3})}|=2|\cos{(\frac{2\pi}{6-\frac{1}{k+1}})}- \cos{(\frac{\pi}{3})}|=1-2\cos{(\frac{2\pi}{6-\frac{1}{k+1}})} & \mbox{ if $n=6k+5$}.\\
\end{array}\right.}
$$
Next we show that $b\leq a$. If $n=6k+1$, then
$$\begin{array}{ll}
a-b&=[2\cos{(\frac{2\pi }{3+\frac{1}{2k}})}+1]-[2\cos{(\frac{2\pi}{6+\frac{1}{k}})}-1]=2\cos{(2\cdot\frac{2\pi}{6+\frac{1}{k}})}-2\cos{(\frac{2\pi}{6+\frac{1}{k}})}+2\\
&=2[2\cos^{2}{(\frac{2\pi}{6+\frac{1}{k}})}-1]-2\cos{(\frac{2\pi}{6+\frac{1}{k}})}+2=2\cos{(\frac{2\pi}{6+\frac{1}{k}})}[2\cos{(\frac{2\pi}{6+\frac{1}{k}})}-1]>0.
\end{array}$$
Similarly, one can verify that $b<a$ if $n=6k+5$.
If $n=6k+2$, then
$$\begin{array}{ll}
a-b&=[-2\cos{(\frac{2\pi }{3-\frac{1}{2k+1}})}-1]-[2\cos{(\frac{2\pi}{6+\frac{2}{k}})}-1]=-2\cos{(\frac{2\pi }{3-\frac{1}{2k+1}})}-2\cos{(\frac{2\pi}{6+\frac{2}{k}})}\\
&=-2\cos{(\pi-\frac{2\pi}{6+\frac{2}{k}})}-2\cos{(\frac{2\pi}{6+\frac{2}{k}})}=2\cos{(\frac{2\pi}{6+\frac{2}{k}})}-2\cos{(\frac{2\pi}{6+\frac{2}{k}})}=0.
\end{array}$$
Similarly, one can verify that $b=a$ if $n=6k+4$.
It follows that
$$\gamma_{2n}(C_n\Box P_2)=b=\left\{\begin{array}{ll}
2\cos{(\frac{ 2\pi \lfloor\frac{n}{6}\rfloor}{n})}-1 & \mbox{if\  $n=6k+1$ or $6k+2$},\\
1-2\cos{(\frac{ 2\pi \lceil\frac{n}{6}\rceil}{n})} & \mbox{if\  $n=6k+4$ or $6k+5$}.
\end{array}\right.$$

We complete this proof.
\end{proof}
\begin{remark}\label{rem-5}
In Remark \ref{rem-2} we mention that the lower bound of $QE(G)$ depend on $\gamma_n$. The Lemma \ref{lem-20} provides us examples that there exist a sequence of graphs such that $\gamma_{2n}=0$, say $\gamma_{6k}(C_{3k}\Box P_2)=0$ for any $k>0$ and also exist a sequence of graphs  such that $\gamma_{2n}>0$, say $\gamma_{6k}(C_{3k}\Box P_2)=b>0$ for any $k>0$. However, $b$ tends to zero while $k$ goes to infinite. It is interesting to find the sequence of graphs $\{G_n\}$ such that there exists a constant $c>0$ satisfying $\gamma_n>c$.
\end{remark}

From Lemma \ref{lem-20}, we get the following result.
\begin{thm}\label{exa-2}
Let $G=C_n\Box P_2$ be a $n$-prism with $2n\geq6$ vertices. Then
\begin{equation}\\\\\label{EQ-eq-18}
\footnotesize{3+\sqrt{3(2n-1)(2n-3)}>
E(G)=QE(G) \geq \left\{\begin{array}{ll}
{2n} & \mbox{if\ $n=3k$},\\
{6n \cdot \frac{\sqrt{{2\cos(\frac{2\pi\lfloor\frac{n}{6}\rfloor}{n})-1}}}{1+\cos(\frac{2\pi\lfloor\frac{n}{6}\rfloor}{n})}} & \mbox{if\ $n=6k+1$ or $6k+2$},\\
{6n \cdot \frac{\sqrt{1-2\cos(\frac{2\pi\lceil\frac{n}{6}\rceil}{n})}}{2-\cos(\frac{2\pi\lceil\frac{n}{6}\rceil}{n})}} & \mbox{if\ $n=6k+4$ or $6k+5$}
\end{array}\right.}
\end{equation}
with the right equality holds if and only if $G$ is $4$-prism $(C_{4}\Box P_2)$.
\end{thm}

\begin{proof}
By Corollary \ref{cor-7}, we have $E(G)=QE(G)<3+\sqrt{3(2n-1)(2n-3)}$.
By Corollary \ref{cor-3}, we have
$$E(G)=QE(G) \left\{\begin{array}{ll}
{>2n} & \mbox{ if $\gamma_{n} = 0$},\\
{\geq 12n \cdot \frac{\sqrt{\gamma_{2n}}}{3+\gamma_{2n}}} & \mbox{ if $\gamma_{2n} > 0$ with equality iff $G \cong K_{4,4}\backslash F= C_{4}\Box P_2$}.
\end{array}\right.$$
Since $\frac{\sqrt{\gamma_{2n}}}{3+\gamma_{2n}}$ is increased on $\gamma_{2n}\in (0,3)$,
we get the required results by Lemma \ref{lem-19}.
\end{proof}

\begin{remark}\label{rem-6}
First by  putting $n=6k+1$, we have
$$\small{\begin{array}{ll}
\lim_{n\to \infty} 6\sqrt{n} \cdot \frac{\sqrt{{2\cos(\frac{2\pi\lfloor\frac{n}{6}\rfloor}{n})-1}}}{1+\cos(\frac{2\pi\lfloor\frac{n}{6}\rfloor}{n})}
&=\lim_{k\to \infty} 4\sqrt{6k+1}\cdot\sqrt{2\cos(\frac{2\pi k}{6k+1})-1}
=\lim_{k\to \infty} 4\sqrt{\frac{1-2\cos(\frac{2\pi k}{6k+1})}{\frac{1}{6k+1}}}\\
&=\lim_{k\to \infty} 4\sqrt{\frac{4\pi\sin{(\frac{2\pi k}{6k+1})}}{3}}
=4\cdot3^{-\frac{1}{4}}\sqrt{2\pi}.
\end{array}}$$
Similarly, $\lim_{n\to \infty} 6\sqrt{n} \cdot \frac{\sqrt{{2\cos(\frac{2\pi\lfloor\frac{n}{6}\rfloor}{n})-1}}}{1+\cos(\frac{2\pi\lfloor\frac{n}{6}\rfloor}{n})}=4\cdot3^{-\frac{1}{4}}\sqrt{2\pi}$ if $n=6k+2$, and $\lim_{n\to \infty} 6\sqrt{n} \cdot \frac{\sqrt{1-2\cos(\frac{2\pi\lceil\frac{n}{6}\rceil}{n})}}{2-\cos(\frac{2\pi\lceil\frac{n}{6}\rceil}{n})}=4\cdot3^{-\frac{1}{4}}\sqrt{2\pi}$ for the rest of $n$.
It implies that $\frac{QE(C_n\Box P_2)}{\sqrt{n}}>4\cdot3^{-\frac{1}{4}}\sqrt{2\pi}$ while $n$ is sufficiently large.
\end{remark}

In the following Table 1, we list the values for exact values and lower bounds of $QE(C_n\Box P_2)$ from $2n=6$ to $20$. The column on exact item lists the exact values of $QE(C_n\Box P_2)$ counted by definition of energy, the other columns list the values of lower bounds of $QE(C_n\Box P_2)$ counted by corresponding formula of inequalities labeling from (\ref{EQ-eq-1}) to (\ref{EQ-eq-24}) and (\ref{EQ-eq-18}). It is clear that the values in the column corresponding our formula (the right of (\ref{EQ-eq-18})) are closer to exact values. Similarly, we list in Table 2 such values for upper bounds of $QE(C_n\Box P_2)$, where the values in the column corresponding our formula (the left of (\ref{EQ-eq-18})) are closer to exact values.

\begin{table}[H]
\caption{ The exact value of $QE(C_n\Box P_2)$ and some lower bounds of $QE(C_n\Box P_2)$}
\begin{tabular*}{15cm}{p{40pt}|p{41pt}p{41pt}p{41pt}p{41pt}p{41pt}p{41pt}p{62pt}}
\hline
\diagbox [width=3.8em,trim=l] {$~~~~2n$}{\scriptsize$QE(G)$}& ~\textcolor[rgb]{1.00,0.00,0.00}{Exact} & $~~~(\ref{EQ-eq-1})$ & $~~~(\ref{EQ-eq-2})$ & $~~~(\ref{EQ-eq-3})$& $~~~(\ref{EQ-eq-4})$ & $~~~(\ref{EQ-eq-5})$ & \textcolor[rgb]{1.00,0.00,0.00}{\footnotesize{Right of $(\ref{EQ-eq-18})$}} \\\hline
~~~~6   &\textcolor[rgb]{1.00,0.00,0.00}{8.0000 }   &6.0000    &2.0000  &3.4641  &4.0000 &0.0000  &~\textcolor[rgb]{1.00,0.00,0.00}{6.0000}       \\\hline
~~~~8   &\textcolor[rgb]{1.00,0.00,0.00}{12.0000}   &6.0000    &2.0000  &3.4641  &4.0000 &0.0000  &~\textcolor[rgb]{1.00,0.00,0.00}{12.0000}       \\\hline
~~~~10  &\textcolor[rgb]{1.00,0.00,0.00}{14.4721}   &6.0000    &2.0000  &3.4641  &4.0000 &0.0000  &~\textcolor[rgb]{1.00,0.00,0.00}{10.9646}      \\\hline
~~~~12  &\textcolor[rgb]{1.00,0.00,0.00}{16.0000}   &6.0000    &2.0000  &3.4641  &4.0000 &0.0000  &~\textcolor[rgb]{1.00,0.00,0.00}{12.0000}       \\\hline
~~~~14  &\textcolor[rgb]{1.00,0.00,0.00}{20.1957}   &6.0000    &2.0000  &3.4641  &4.0000 &0.0000  &~\textcolor[rgb]{1.00,0.00,0.00}{12.8567 }      \\\hline
~~~~16  &\textcolor[rgb]{1.00,0.00,0.00}{23.3137}   &6.0000    &2.0000  &3.4641  &4.0000 &0.0000  &~\textcolor[rgb]{1.00,0.00,0.00}{18.0964}       \\\hline
~~~~18  &\textcolor[rgb]{1.00,0.00,0.00}{25.6459}   &6.0000    &2.0000  &3.4641  &4.0000 &0.0000  &~\textcolor[rgb]{1.00,0.00,0.00}{18.0000}       \\\hline
~~~~20  &\textcolor[rgb]{1.00,0.00,0.00}{28.9443}   &6.0000    &2.0000  &3.4641  &4.0000 &0.0000  &~\textcolor[rgb]{1.00,0.00,0.00}{21.9293}     \\\hline

\end{tabular*}
\end{table}

\begin{table}[H]
\caption{ The exact value of $QE(C_n\Box P_2)$ and some upper bounds of $QE(C_n\Box P_2)$}
\begin{tabular*}{15cm}{p{40pt}|p{52pt}p{52pt}p{52pt}p{52pt}p{52pt}p{55pt}}
\hline
\diagbox [width=3.8em,trim=l] {~~~~$2n$}{\scriptsize{$QE(G)$}}& ~\textcolor[rgb]{1.00,0.00,0.00}{Exact} & $~~~~(\ref{EQ-eq-21})$ & $~~~~(\ref{EQ-eq-22})$ & $~~~~(\ref{EQ-eq-23})$& $~~~~(\ref{EQ-eq-24})$ & \textcolor[rgb]{1.00,0.00,0.00}{\footnotesize{Left of $(\ref{EQ-eq-18})$}} \\\hline
~~~~~6   &\textcolor[rgb]{1.00,0.00,0.00}{8.0000}    &30.0000    &21.9017  &23.6499  &26.0000  &~\textcolor[rgb]{1.00,0.00,0.00}{9.7082}        \\\hline
~~~~~8   &\textcolor[rgb]{1.00,0.00,0.00}{12.0000}   &42.0000    &29.3939  &33.5710  &38.0000  &~\textcolor[rgb]{1.00,0.00,0.00}{13.2470 }       \\\hline
~~~~~10  &\textcolor[rgb]{1.00,0.00,0.00}{14.4721}   &54.0000    &36.6449  &45.0379  &50.0000  &~\textcolor[rgb]{1.00,0.00,0.00}{16.7477 }       \\\hline
~~~~~12  &\textcolor[rgb]{1.00,0.00,0.00}{16.0000}   &66.0000    &43.7490  &57.8720  &62.0000  &~\textcolor[rgb]{1.00,0.00,0.00}{20.2337}        \\\hline
~~~~~14  &\textcolor[rgb]{1.00,0.00,0.00}{20.1957}   &78.0000    &50.7504  &71.9243  &74.0000  &~\textcolor[rgb]{1.00,0.00,0.00}{23.7123}        \\\hline
~~~~~16  &\textcolor[rgb]{1.00,0.00,0.00}{23.3137}   &90.0000    &57.6742  &87.0799  &86.0000  &~\textcolor[rgb]{1.00,0.00,0.00}{27.1868}       \\\hline
~~~~~18  &\textcolor[rgb]{1.00,0.00,0.00}{25.6459}   &102.0000   &64.5367  &103.2493 &98.0000  &~\textcolor[rgb]{1.00,0.00,0.00}{30.6586 }      \\\hline
~~~~~20  &\textcolor[rgb]{1.00,0.00,0.00}{28.9443}   &114.0000   &71.3489  &120.3610 &110.0000 &~\textcolor[rgb]{1.00,0.00,0.00}{34.1288}       \\\hline

\end{tabular*}
\end{table}

\end{document}